\documentclass[11pt]{amsart}
\usepackage{tikz-cd}
\usepackage{fancyhdr} 
 \newcommand{\filename}{29-Jan-2025-Trinion-Quantum-Horn.tex}  
\usepackage{latexsym, amssymb, amscd, amsfonts,pstricks,enumitem,amsmath, young}
\usepackage{latexsym, amssymb, amscd, amsfonts,pstricks,enumitem,amsmath}
\usepackage{amsmath,amsthm,amssymb,mathrsfs}
 
\usepackage[all]{xypic}
\usepackage{ifthen}
\usepackage{longtable}
 
\numberwithin{equation}{section}

\providecommand{\binom}[2]{{#1\choose#2}}

\newcommand{\Hom}{\operatorname{Hom}}
\renewcommand{\geq}{\geqslant}
\renewcommand{\leq}{\leqslant}

\newcommand{\K}{\mathrm{K}}                            

\newcommand{\CC}{\mathbb{C}} 
\newcommand{\RR}{\mathbb{R}} 

\newtheorem{theorem}{Theorem}[section]

\theoremstyle{definition}

\newtheorem{problem}[theorem]{Problem}

\begin{document}

\title[The multiplicative $\operatorname{SU}(2)$-eigenvalue problem for the trinion]{About the Bohr-Sommerfeld polytope and the multiplicative $\operatorname{SU}(2)$-eigenvalue problem for the trinion}
\author{Nathan Grieve}

\address{Department of Mathematics \& Statistics,
Acadia University, Huggins Science Hall, Room 130,
12 University Avenue,
Wolfville, NS, B4P 2R6
Canada; School of Mathematics and Statistics, Carleton University, 4302 Herzberg Laboratories, 1125 Colonel By Drive, Ottawa, ON, K1S 5B6, Canada; 
D\'{e}partement de math\'{e}matiques, Universit\'{e} du Qu\'{e}bec \`a Montr\'{e}al, Local PK-5151, 201 Avenue du Pr\'{e}sident-Kennedy, Montr\'{e}al, QC, H2X 3Y7, Canada; Department of Pure Mathematics, University of Waterloo, 200 University Avenue West, Waterloo, ON, N2L 3G1, Canada;  
}
\email{nathan.m.grieve@gmail.com}%

\begin{abstract} We study in detail the work of Jeffrey and Weitsman, especially \cite{Jeffrey:Weitsman:1992} and \cite{Jeffrey:Weitsman:1994}.  In particular, we describe the manner in which the inequalities that cut out the Bohr-Sommerfeld moment polytope give rise to a solution of the $\operatorname{SU}(2)$-multiplicative eigenvalue problem for the case of three conjugacy classes.
\end{abstract}
\thanks{
\emph{Mathematics Subject Classification (2020):} 14D20, 14H60, 58D27. \\
\emph{Key Words: Bohr-Sommerfeld moment polytope, multiplicative eigenvalue problem, trinion representation space.} 
\\
I thank the Natural Sciences and Engineering Research Council of Canada for their support through my grants DGECR-2021-00218 and RGPIN-2021-03821. \\
ORCID: https://orcid.org/0000-0003-3166-0039
\\
Date: \today.  \\
File name: \filename
}

\maketitle

\section{Introduction}

The purpose of the present note, is to describe the manner in which the inequalities that cut out the Bohr-Sommerfeld moment polytope, as discovered by Jeffrey and Weitsman, see \cite{Jeffrey:Weitsman:1992} and \cite{Jeffrey:Weitsman:1994}, give rise to a solution of the $\operatorname{SU}(2)$-multiplicative eigenvalue problem for the case of three conjugacy classes.

In order to help place matters into their proper context, let us first recall, as in \cite{Belkale:2001}, see also \cite{Agnihotri:Woodward:1998} and \cite{Teleman:Woodward:2003}, the statement of the multiplicative eigenvalue problem for the $n \times n$ special unitary group 
$$
\operatorname{SU}(n) := \left\{A  \in \CC^{n \times n} : AA^* = 1 \text{ and } \operatorname{det}(A) = 1 \right\} \text{.}
$$
Here, $A^*$ denotes the conjugate transpose of $A$. 

\begin{problem}[{\cite{Belkale:2001}, \cite{Agnihotri:Woodward:1998}, \cite{Teleman:Woodward:2003}}]\label{Multiplicative:SU:n:Eigen:Value:Problem}
Characterize the possible eigenvalues of those matrices 
$$
A_{1},\dots,A_{s} \in \operatorname{SU}(n)
$$
which satisfy the condition that
$$
A_{1} \cdot \hdots \cdot A_{s} = 1 \text{.}
$$
\end{problem}

Second, here we state \cite[Theorem 1.4.3]{Belkale:2022} as Theorem \ref{Belkale:multiplicative:horn} below.  A conceptual formulation is that the multiplicative $\operatorname{SU}(n)$-eigenvalue problem is controlled by the Grassmannian varieties' quantum Schubert calculus.  We do not pursue here a more detailed discussion of the quantum Schubert calculus and its interpretation in terms of the quantum Gromov-Witten intersection numbers.  However, we do mention that these quantum Gromov-Witten numbers arise in the statement of Theorem \ref{Belkale:multiplicative:horn}.  More details about the quantum Schubert calculus can be found in \cite[Sections 1.3 and 2.4]{Belkale:2008} and \cite{Bertram:Ciocan-Fontanine:Fulton:1999}, for example (among others). 

Before stating Theorem \ref{Belkale:multiplicative:horn} below, denote by
$$
\Delta(n) \subseteq \RR^n 
$$
the standard $n-1$-simplex.  Recall, that it consists of those
$$
\alpha = (\alpha_1,\dots,\alpha_n) \in \RR^n
$$
which have the two properties that
\begin{enumerate}
\item[(i)]{
$
\alpha_1 \geq \dots \geq \alpha_n \geq \alpha_1 - 1
$;
and
}
\item[(ii)]{
$\sum_{i=1}^n \alpha_\ell = 0$.
}
\end{enumerate}
Further, recall that $\Delta(n)$ may be regarded as the parameter space for $\operatorname{SU}(n)$-conjugacy classes by associating to each
$$
\alpha = (\alpha_1,\dots,\alpha_n) \in \Delta(n)
$$
the conjugacy class of the matrix
$$
\operatorname{diag}\left(\exp(2 \pi \sqrt{-1}\alpha_1),\dots,\exp(2 \pi \sqrt{-1}\alpha_n)\right) \in \operatorname{SU}(n) \text{.}
$$
Having fixed some notation and context, the main result about the multiplicative $\operatorname{SU}(n)$-eigenvalue problem can be formulated in the following way.

\begin{theorem}[Compare with {\cite[Theorem 1.1]{Belkale:2008}, \cite[Theorem 1.4.3]{Belkale:2022}}]\label{Belkale:multiplicative:horn}
Fix a collection of $\operatorname{SU}(n)$-conjugacy classes
$$
\alpha_1,\dots,\alpha_s 
$$
with
$$
\alpha_j = (\alpha_{1j},\dots,\alpha_{nj}) \in \Delta(n) \text{,}
$$
for $j = 1,\dots,s$.  Then, the following two assertions are equivalent.
\begin{enumerate}
\item[(i)]{
There exist matrices
$$
A_{1},\dots,A_{s} \in \operatorname{SU}(n)
$$
with $A_{j}$ in the conjugacy class $\alpha_j$ and such that 
$$
A_{1}\cdot \hdots \cdot A_{s} = 1 \text{.}
$$
}
\item[(ii)]{
For all integers $r,d$, with $0 < r < n$, $d \geq 0$ and  all $s$-tuples of cardinality $r$ subsets $I_1,\dots,I_s \in \binom{[n]}{r}$, for $[n] :=\{1,\dots,n\}$, such that the quantum Gromov-Witten intersection number $\langle \sigma_{I_1},\dots, \sigma_{I_s} \rangle_d = 1$, the following inequality is valid
$$
\sum_{j=1}^s \sum_{k \in I_j} \alpha_{kj} \leq d \text{.}
$$
}
\end{enumerate}
\end{theorem} 

Here, we study work of Jeffrey and Weitsman, especially \cite{Jeffrey:Weitsman:1992} and \cite{Jeffrey:Weitsman:1994}.  In doing so, we obtain, via the $\operatorname{SU}(2)$-representation theory of the trinion, a solution to the smallest nontrivial instance of  Problem \ref{Multiplicative:SU:n:Eigen:Value:Problem}.  This is the content of Theorem \ref{SU:2:Trinion:eigenvalue:thm} below.  While, in many ways this result is fairly well known, it remains instructive to revisit it in detail.

Prior to our formulation of Theorem \ref{SU:2:Trinion:eigenvalue:thm}, let us first recall,  that the set of $\operatorname{SU}(2)$-conjugacy classes is in bijective correspondence with the set of $2 \times 2$ special unitary matrices
$$
\left\{C(\theta) := 
 \left( 
 \begin{matrix}
 \cos(\theta) + \sqrt{-1} \sin(\theta) & 0 \\
 0 & \cos(\theta) - \sqrt{-1} \sin(\theta)
 \end{matrix}
 \right) : \theta \in [0,\pi] \right\} \text{.}
$$

Having fixed this notation, Theorem \ref{SU:2:Trinion:eigenvalue:thm} may be formulated in the following way.

\begin{theorem}\label{SU:2:Trinion:eigenvalue:thm}
Fix three $\operatorname{SU}(2)$-conjugacy classes $\mathrm{C}(\theta_1), \mathrm{C}(\theta_2), \mathrm{C}(\theta_3)$ corresponding to 
$$\theta_i \in [0,\pi] \text{ for $i = 1,2,3$.}$$ 
Then, there exists a collection of $2 \times 2$ special unitary matrices 
$$A_1, A_2, A_3 \in \operatorname{SU}(2)$$ 
which satisfy the two conditions that
\begin{enumerate}
\item[(i)]{
$A_i \in \mathrm{C}(\theta_i)$, for $i = 1,2,3$; and
}
\item[(ii)]{
$A_1 \cdot A_2 \cdot A_3 = \left( \begin{matrix}
1 & 0 \\
0 & 1
\end{matrix}
\right) \text{;} $
}
\end{enumerate}
if and only if there exists a three tuple
$$
\overrightarrow{\theta} = (\theta_1,\theta_2,\theta_3) \in [0,\pi]^3 
$$
which lies in the tetrahedron
$$
\Delta \subseteq \RR^3
$$
which is inscribed in the cube
$$
\left\{ \overrightarrow{\theta} = (\theta_1, \theta_2, \theta_3) : 0 \leq \theta_i \leq \pi \text{ for $i = 1,2,3$} \right\} \subseteq \RR^3
$$
and, in particular, which is described by the following four inequalities
\begin{itemize}
\item{$\theta_1 - \theta_2 \leq \theta_3$; }
\item{$-\theta_1 + \theta_2 \leq \theta_3$; }
\item{ $\theta_3 \leq \theta_1 + \theta_2$; and}
\item{ $\theta_3 \leq 2 \pi - (\theta_1 + \theta_2)$\text{.}}
\end{itemize}
\end{theorem}

The proof of Theorem \ref{SU:2:Trinion:eigenvalue:thm} is based on the result of Jeffrey and Weitsman \cite[Proposition 3.1]{Jeffrey:Weitsman:1992}.  Indeed, as a complementary preliminary result, in Theorem \ref{SU:2:Trinion:Moment:Map:Thm} we described the bijective correspondence between the set of equivalence classes of representations 
$$
\rho \colon \pi_1 \rightarrow \operatorname{SU}(2)
$$
and the image of the moment map
$$
\theta : \mathcal{M} := \Hom(\pi_1, \operatorname{SU}(2)) / \operatorname{SU}(2) \rightarrow [0,\pi]^3 \text{.}
$$
Theorem \ref{SU:2:Trinion:eigenvalue:thm} is then deduced as an application of Theorem \ref{SU:2:Trinion:Moment:Map:Thm}.  As an additional comment, we note that Theorems \ref{Belkale:multiplicative:horn} and \ref{SU:2:Trinion:eigenvalue:thm} complement \cite[Section 3 (ii)]{Hurtubise:Jeffrey:2000}.

\subsection*{Acknowledgements}  It is my pleasure to thank colleagues for their interest in this work and for discussions and correspondence on related topics.  Further, I thank the Natural Sciences and Engineering Research Council of Canada for their support through my grants DGECR-2021-00218 and RGPIN-2021-03821.

\section{Preliminaries about the $\operatorname{SU}(2)$-trinion representation space}
We fix our conventions about the three punctured sphere and its $\operatorname{SU}(2)$-representation space.  They resemble closely the approach from \cite{Jeffrey:Weitsman:1992}, \cite{Jeffrey:Weitsman:1994} and \cite{Goldman:1984}.

By a \emph{trinion}, is meant a copy of the two-punctured disc
$$
D = \{z \in \CC : |z|<2\} \setminus \left( \{z : |z-1| < 1/2 \} \bigcup \{ z : |z+1| < 1/2 \} \right)
$$
together with marked points on the boundary components of $D$; the standard orientation is induced from that of the complex plane $\CC$.  Each such  trinion is homeomorphic to the three punctured sphere.

Consider the group $\pi_1$, that is generated by the three symbols $c_1,c_2,c_3$ and subject to the defining relation that 
$$
c_1 \cdot c_2 \cdot c_3 = 1 \text{.}
$$
Then $\pi_1$ is the fundamental group of the \emph{trinion}.  Moreover, consider its $\operatorname{SU}(2)$-representation space
$$
\mathcal{M} := \Hom(\pi_1, \operatorname{SU}(2)) / \operatorname{SU}(2) \text{.}
$$
Here, the quotient is taken with respect to conjugation.

For $i = 1,2,3$, define the angle functions
$$
\theta_i(\cdot) \colon \Hom(\pi_1,\operatorname{SU}(2)) \rightarrow [0,\pi]
$$
by the condition that
$$
\theta_i(\rho) = \operatorname{arccos}\left( \frac{1}{2} \operatorname{Tr} \rho(c_i)\right) \text{.}
$$
These functions $\theta_i(\cdot)$ are well-defined modulo the conjugation action of $\operatorname{SU}(2)$ on $\Hom(\pi_1,\operatorname{SU}(2))$.

By abuse of notation, denote the induced maps by
$$
\theta_i(\cdot) \colon \mathcal{M} \rightarrow [0,\pi] \text{.}
$$
Put
\begin{equation}\label{moment:map}
\theta = (\theta_1(\cdot),\theta_2(\cdot),\theta_3(\cdot)) \colon \mathcal{M} \rightarrow [0,\pi]^3 \text{.}
\end{equation}
Our aim, in what follows, is to describe the image of the map \eqref{moment:map}.  In doing so, we will improve our understanding of the following interesting result from \cite{Jeffrey:Weitsman:1992} (see also \cite{Jeffrey:Weitsman:1994}).
 
\begin{theorem}\label{SU:2:Trinion:Moment:Map:Thm}
Let $\pi_1$ be the group that is generated by the three symbols $c_1, c_2, c_3$ and subject to the defining relation that 
$$
c_1 \cdot c_2 \cdot c_3 = 1 \text{.}
$$
Then there is a bijective correspondence between the set of equivalence classes of representations 
$$
\rho \colon \pi_1 \rightarrow \operatorname{SU}(2) 
$$
and the image of the moment map
$$
\theta = (\theta_1(\cdot), \theta_2(\cdot), \theta_3(\cdot)) \colon \mathcal{M} \rightarrow [0,\pi]^3 \text{,}
$$
$$
\theta_i([\rho]) = \operatorname{arccos}\left( \frac{1}{2} \operatorname{Tr} \rho(c_i)\right) \text{.}
$$
for
$$
[\rho] \in \mathcal{M} := \Hom(\pi_1, \operatorname{SU}(2)) / \operatorname{SU}(2)
$$
the equivalence class of a given representation 
$$
\rho \in \Hom(\pi_1, \operatorname{SU}(2)) \text{.}
$$
\end{theorem}

We prove Theorem \ref{SU:2:Trinion:Moment:Map:Thm} in Section \ref{image:moment:map}.  But first, 
in order to fix the main idea for what follows, note that if
$$
g_i := 
\left(
\begin{matrix}
\cos(\theta_i) + \sqrt{-1} \sin(\theta_i) & 0 \\
0 & \cos(\theta_i) - \sqrt{-1} \sin(\theta_i) 
\end{matrix}
\right) \text{, }
$$
for $i = 1,2,3$, then
\begin{align*}
\theta_i & = \arccos \left( \frac{1}{2} \operatorname{Tr}(g_i) \right) \\
& = \arccos \left( \frac{1}{2}(2 \cos(\theta_i) ) \right) \text{.}
\end{align*}
Thus, the \emph{angle functions} encode the \emph{character} of the \emph{conjugacy classes} of representations 
$$\rho \in \Hom(\pi_1, \operatorname{SU}(2)) \text{.} $$ 

Finally, we remark that Theorem \ref{SU:2:Trinion:Moment:Map:Thm} implies Theorem \ref{SU:2:Trinion:eigenvalue:thm}.

\begin{proof}[Proof of Theorem \ref{SU:2:Trinion:eigenvalue:thm}]
Because of the defining relations of the trinion fundamental group $\pi_1$,  there exists a bijective correspondence between the conjugacy class of a representation
$$
\rho \colon \pi_1 \rightarrow \operatorname{SU}(2)
$$
and data that consists of three conjugacy classes $\mathrm{C}(\theta_1), \mathrm{C}(\theta_2), \mathrm{C}(\theta_3)$, corresponding to 
$$\theta_i \in [0,\pi]\text{,}$$ 
for $i=1,2,3$, together with matrices
$$
A_1,A_2,A_3 \in \operatorname{SU}(2) \text{,}
$$
which have the property that 
$$A_i \in \mathrm{C}(\theta_i) \text{,} $$
for $i = 1,\dots,3$, and 
$$
A_1 \cdot A_2 \cdot A_3 \text{.}
$$
Indeed, this bijective correspondence is set up, by setting $\mathrm{C}(\theta_i)$ to be the conjugacy class of the matrix
$$
A_i := \rho(\mathrm{c}_i) \text{,}
$$
for $i = 1,2,3$.
\end{proof}

\section{Image of the moment map}\label{image:moment:map}

As will become apparent in the proof of Theorem \ref{SU:2:Trinion:Moment:Map:Thm}, the image of the moment map
$$
\theta  = (\theta_1,\theta_2,\theta_3) \colon \mathcal{M} \rightarrow [0,\pi]^3 
$$
is the tetrahedron that is described by the following four inequalities
\begin{itemize}
\item{
$
\theta_1 - \theta_2 \leq \theta_3 \text{;}
$
}
\item{
$
- \theta_1 + \theta_2 \leq \theta_3 \text{;}
$
}
\item{
$
\theta_3 \leq \theta_1 + \theta_2 \text{; and}
$
}
\item{
$
\theta_3 \leq 2\pi - (\theta_1 + \theta_2) \text{.}
$
}
\end{itemize}
Moreover, as noted in \cite[Lemma 3.5]{Jeffrey:Weitsman:1992}, under the moment map $\theta$, the reducible representations correspond to those 
$$
\theta = (\theta_1,\theta_2,\theta_3) \in [0,\pi]^3
$$
which fail to lie in the interior of the moment map; such three tuples satisfy at least one of the following equations
\begin{itemize}
\item{
$\theta_1 + \theta_2 - \theta_3 = 0$;
}
\item{
$\theta_2 + \theta_3 - \theta_1 = 0$;
}
\item{
$\theta_3 + \theta_1 - \theta_2 = 0$; and
}
\item{
$\theta_1 + \theta_2 + \theta_3 = 2 \pi$.
}
\end{itemize}

\begin{proof}[Proof of Theorem \ref{SU:2:Trinion:Moment:Map:Thm}]

The proof of Theorem \ref{SU:2:Trinion:Moment:Map:Thm} follows, closely, that of \cite[Proposition 3.1]{Jeffrey:Weitsman:1992}.  Specifically, the strategy of proof is to derive the condition on a triple of angles
$$
\theta = (\theta_1,\theta_2,\theta_3) \in [0,\pi]^3 \text{,}
$$
so that they arise as \emph{holonomy angles} of some flat connection on the trinion.

In particular, in what follows, we study the conjugacy classes
$$
\mathrm{C}(\theta_i) \in \operatorname{SU}(2) \text{.}
$$
We want to study the condition on \emph{the angles} $\theta_i$, for $i = 1,2,3$, for there to exist
$$
g_i \in \mathrm{C}(\theta_i)
$$
which satisfy the condition that $g_1 g_2$ is conjugate to $g_3$.  

By conjugation, without loss of generality, we may assume that
$$
g_1 :=
\left(
\begin{matrix}
\cos(\theta_1) + \sqrt{-1} \sin(\theta_1) & 0 \\
0 & \cos(\theta_1) - \sqrt{-1} \sin(\theta_1)
\end{matrix}
\right) \in \mathrm{C}(\theta_1)
$$
and
$$
g_2 := 
\left(
\begin{matrix}
\cos(\theta_2) + \sqrt{-1} \sin(\theta_2)\cos(\beta) & \sin(\theta_2)\sin(\beta) \\
-\sin(\theta_2)\sin(\beta) & \cos(\theta_2) - \sqrt{-1} \sin(\theta_2)\cos(\beta) 
\end{matrix}
\right) \in \mathrm{C}(\theta_2) \text{, }
$$
where 
$$\beta \in [0,2\pi] \text{.}$$

The condition of interest is that
$$
g_1 \cdot g_2 \in \mathrm{C}(\theta_3) \text{.}
$$

Now, observe  that
\begin{align*}
g_1 \cdot g_2 & = 
\left(
\begin{matrix}
a & b \\
c & d
\end{matrix}
\right)
\end{align*}
where
$$
a = \cos(\theta_1)\cos(\theta_2) - \sin(\theta_1)\sin(\theta_2)\cos(\beta) + \sqrt{-1} (\sin(\theta_1)\cos(\theta_2) + \cos(\theta_1)\sin(\theta_2)\cos(\beta) ) \text{;}
$$
$$
b = \cos(\theta_1)\sin(\theta_2)\sin(\beta) + \sqrt{-1} \sin(\theta_1)\sin(\theta_2)\sin(\beta) \text{;}
$$
$$
c = - \cos(\theta_1)\sin(\theta_2)\sin(\beta) + \sqrt{-1} \sin(\theta_1)\sin(\theta_2)\sin(\beta) \text{;}
$$
and
$$
d = \cos(\theta_1)\cos(\theta_2) - \sin(\theta_1)\sin(\theta_2)\cos(\beta) - \sqrt{-1} ( \sin(\theta_1)\cos(\theta_2) + \cos(\theta_1)\sin(\theta_2) \cos(\beta)) \text{.}
$$

It follows that 
$$
g_1 \cdot g_2 \cdot g_3 = 1
$$
for
$$
g_3 = \frac{1}{ad-bc} \cdot \left( 
\begin{matrix}
d & - b \\
-c & a
\end{matrix}
\right) \text{.}
$$

Moreover, 
the characteristic polynomial of the matrix $g_1 \cdot g_2$ is 
$$
T^2 - 2 (\cos(\theta_1)\cos(\theta_2) - \sin(\theta_1)\sin(\theta_2)\cos(\beta)) T + 1 \text{.}
$$
In light of these calculations, the condition that 
$$
g_1 \in \mathrm{C}(\theta_1) \text{ and } g_2 \in \mathrm{C}(\theta_2)
$$
have the property that
$$
g_1 \cdot g_2 \in \mathrm{C}(\theta_3)
$$
becomes the condition that
$$
\cos(\theta_1)\cos(\theta_2) - \sin(\theta_1)\sin(\theta_2)\cos(\beta) = \cos(\theta_3) \text{.}
$$

Evidently, since 
$$\cos(\beta) \in [-1,1]\text{,}$$ 
the above equation admits a solution for 
$$\beta \in [0, 2\pi]$$ 
if and only if the following two inequalities 
$$
\cos(\theta_3) \leq \sin(\theta_1)\sin(\theta_2) + \cos(\theta_1)\cos(\theta_2)
$$
and
$$
\cos(\theta_1)\cos(\theta_2) - \sin(\theta_1)\sin(\theta_2) \leq \cos(\theta_3)
$$
hold true.

By the cosine addition identities, these two inequalities may be expressed more succinctly as
$$
\cos(\theta_1 + \theta_2) \leq \cos(\theta_3) \leq \cos(\theta_1 - \theta_2) \text{.}
$$

Recall, we are looking for angles 
$$\theta_i \in [0,\pi] \text{.}$$  
Moreover, $\cos(\cdot)$ is a decreasing function on $[0,\pi]$.  Thus, the above identity 
$$
\cos(\theta_1 + \theta_2) \leq \cos(\theta_3) \leq \cos(\theta_1 - \theta_2)
$$
may equivalently be expressed as the three conditions that
$$
|\theta_1 - \theta_2| \leq \theta_3 \text{;}
$$
$$
\theta_3 \leq \theta_1 + \theta_2 \text{ if $|\theta_1 + \theta_2| \leq \pi$; and}
$$
$$
\theta_3 \leq 2 \pi - (\theta_1 + \theta_2) \text{ if $|\theta_1 + \theta_2| > \pi$.}
$$
Indeed, the first condition arises because of the constraint 
$$\cos(\theta_3) \leq \cos(\theta_1 - \theta_2)$$ 
whereas the last two conditions arise because of the constraint
$$
\cos(\theta_1 + \theta_2) \leq \cos(\theta_3) \text{.}
$$

The holonomy condition is thus
$$
|\theta_1 - \theta_2 | \leq \theta_3 \leq \min \{ \theta_1 + \theta_2, 2 \pi - (\theta_1 + \theta_2)  \} \text{.}
$$
It defines a tetrahedron that is inscribed in the cube 
$$0 \leq \theta_i \leq \pi \text{,}$$
for $i = 1,2,3$.

Finally, let us establish that the map
$$
\theta = (\theta_1(\cdot), \theta_2(\cdot), \theta_3(\cdot)) \colon \mathcal{M} \rightarrow [0,\pi]^3 
$$
is injective.  

First, in the above discussion, we have implicitly chosen 
$$
\sin(\theta_2) \sin(\beta) \geq 0 \text{.}
$$
Moreover, 
$$
\sin(\theta_2) \sin(\beta) \geq 0
$$
is uniquely determined if
$$
\sin(\theta_1) \not = 0 \text{.}
$$

Thus, the given equation for $g_2$ has a unique solution in terms of $\cos(\theta_1)$, $\cos(\theta_2)$ and $\cos(\theta_3)$.  This remains true in the degenerate case that 
$$
\sin(\theta_1) = 0 \text{.}
$$

In particular, by the above discussion, the holonomy functions $\theta_i(\cdot)$, for $i = 1,2,3$, may be considered as coordinates on the space of conjugacy classes of representations of the trinion fundamental group.

Finally, observe that the equation
$$
|\theta_1 - \theta_2| \leq \theta_3 \leq \min \{ \theta_1 + \theta_2, 2\pi - (\theta_1 + \theta_2) \}
$$
defines a tetrahedron that is inscribed in the cube
$$
0 \leq \theta_i \leq \pi \text{, for $i = 1,2,3$.}
$$
This tetrahedron has four vertices 
\begin{itemize}
\item{
$S = (0,0,0)$;
}
\item{
$R = (\pi,\pi,0)$;
}
\item{
$Q = (0,\pi,\pi)$; and
} 
\item{
$P = (\pi,0,\pi)$.  
}
\end{itemize}
It is bounded by the four affine hyperplanes
\begin{itemize}
\item{
$
\theta_3 - \theta_1 + \theta_2 = 0 \text{;}
$
}
\item{
$
\theta_3 + \theta_1 - \theta_2 = 0 \text{;}
$
}
\item{
$
\theta_3 - \theta_1 - \theta_2 = 0 \text{; and}
$
}
\item{
$
\theta_3 + \theta_1 + \theta_2 = 2 \pi \text{.}
$
}
\end{itemize}

Rewriting the above description of the tetrahedron leads to the following four inequalities
\begin{itemize}
\item{
$
\theta_1 - \theta_2 \leq \theta_3 \text{;}
$
}
\item{
$
- \theta_1 + \theta_2 \leq \theta_3 \text{;}
$
}
\item{
$
\theta_3 \leq \theta_1 + \theta_2 \text{; and}
$
}
\item{
$
\theta_3 \leq 2\pi - (\theta_1 + \theta_2) \text{.}
$
}
\end{itemize}
This completes the proof of Theorem  \ref{SU:2:Trinion:Moment:Map:Thm}.
\end{proof}

\section{The kernel of the Hamiltonian flows}
We now describe the moment map picture from \cite{Jeffrey:Weitsman:1992} and \cite{Jeffrey:Weitsman:1994}.  In keeping with the above notation, put
$$
h_j := \frac{1}{\pi} \theta_j(\cdot) \text{,}
$$
for $j = 1,2,3$.  The \emph{Hamiltonian flows} $\Phi^{h_j}_t(x)$ of the functions $h_j$, at time $t$, and starting at the point $x$, are defined on an open dense set $U$.  It is known that these functions Poisson commute on $U$.

There is an induced action of $\RR^3$ on $U$ that is given by
$$
((\lambda_1, \lambda_2, \lambda_3),x) \mapsto \Phi_{t = 1}^{\sum_j \lambda_j h_j }(x) \text{.}
$$
The \emph{kernel} of this action is the subgroup of $\RR^3$ that acts trivially for all $x$.  It is a lattice $\Lambda$ in $\RR^3$.  There is thus an induced action of a \emph{torus} 
$$\K = \RR^3 / \Lambda$$ 
on $U$.  

In order to identify the lattice, it is important to have a description of the \emph{Hamiltonian flows}.   But, first, the $\operatorname{SU}(2)$-representation space $\mathcal{M}$ is in bijective correspondence with the set 
$$
\{(t_1,t_2,t_3) : [0,1]^3 : | t_1 - t_2 | \leq t_3 \leq t_1 + t_2 \text{, } t_1 + t_2 + t_3 \leq 2 \} \text{.}
$$
This correspondence is induced by the map
$$
x \mapsto (h_1(x),h_2(x),h_3(x))
$$
where the 
$$h_j \left( \cdot \right)= h_{c_j} \left(\cdot \right) \colon \mathcal{M} \rightarrow [0,1]$$ 
are the functions on $\mathcal{M}$ that are defined by the conditions that
$$
f_{c_j} (\cdot) = 2 \cos \left( \pi h_j (\cdot) \right) \text{.}
$$

It follows that 
$$
\overline{\mu(U)} \subseteq \RR^3
$$
the closure of the image of
$$
U  = \bigcap_j h_j^{-1}((0,1))
$$
under the moment map $\mu$ is a convex polyhedron.

Each boundary circle $c_j$ represents a copy of the unitary group $U(1)$.  More precisely, there is an identification
$$
U(1) = \{\exp\left(2 \pi \sqrt{-1} \phi_j\right) : 0 \leq \phi_j \leq 1 \} \text{,}
$$
which allows $U(1)$ to be realized as the Hamiltonian flow of the function $h_j$.

The lattice
$$
\Lambda \subseteq \RR^3 \text{,}
$$
that is defined to be the kernel of the Hamiltonian flow has rank $3$.  It is spanned by the four vectors
$$
\mathbf{e}_1 = (1,0,0) \text{, } \mathbf{e}_2 = (0,1,0) \text{, } \mathbf{e}_3 = (0,0,1) \text{ and } \mathbf{e}_0 = \frac{1}{2}\left( \mathbf{e}_1, \mathbf{e}_2, \mathbf{e}_3 \right) = \left( \frac{1}{2}, \frac{1}{2}, \frac{1}{2} \right) \text{.}
$$

The map $\mu$ is thus the moment map for the torus 
$$
\K = \RR^3 / \Lambda \text{.}
$$

\providecommand{\bysame}{\leavevmode\hbox to3em{\hrulefill}\thinspace}
\providecommand{\MR}{\relax\ifhmode\unskip\space\fi MR }
\providecommand{\MRhref}[2]{%
  \href{http://www.ams.org/mathscinet-getitem?mr=#1}{#2}
}
\providecommand{\href}[2]{#2}

\end{document}